\documentclass[a4paper]{article}
\usepackage{url}
\usepackage{amsmath,amsfonts,amssymb,amsthm,graphicx,enumerate,mathrsfs,color,subfigure,multicol,mathdots,booktabs,xypic}
\usepackage[noend]{algpseudocode}
\usepackage{algorithmicx,algorithm}
\newtheorem{definition}{Definition}
\newtheorem{remark}{Remark}
\newtheorem{theorem}{Theorem}
\newtheorem{corollary}{Corollary}
\newtheorem{lemma}{Lemma}

\newcommand{\ten}[1]{\mathcal{#1}}
\newcommand{\tens}[1]{\mathcal{#1}}

\newcommand{\bcirc}{{\rm{bcirc}}}
\newcommand{\unfold}{{\rm{unfold}}}
\newcommand{\fold}{{\rm{fold}}}
\newcommand{\ind}{{\rm{Ind}}}
\newcommand{\rank}{{\rm{rank}}}
\newcommand{\norm}[1]{\left\lVert#1\right\rVert}
\begin{document} \large
\title{T-Jordan Canonical Form and T-Drazin Inverse based on the T-Product}
\author{Yun Miao\footnote{E-mail: 15110180014@fudan.edu.cn. School of Mathematical Sciences, Fudan University, Shanghai, 200433, P. R. of China. Y. Miao is supported by the National Natural Science Foundation of China under grant 11771099. } \quad  Liqun Qi\footnote{ E-mail: maqilq@polyu.edu.hk. Department of Applied Mathematics, the Hong Kong Polytechnic University, Hong Kong. L. Qi is supported by the Hong Kong Research Grant Council (Grant No. PolyU 15302114, 15300715, 15301716 and 15300717)} \quad  Yimin Wei\footnote{Corresponding author. E-mail: ymwei@fudan.edu.cn and yimin.wei@gmail.com. School of Mathematical Sciences and Shanghai Key Laboratory of Contemporary
Applied Mathematics, Fudan University, Shanghai, 200433, P. R. of China. Y. Wei is supported by the Innovation Program of Shanghai Municipal Education Commission.
}}
\maketitle

\begin{abstract}
In this paper, we investigate the tensor similar relationship and propose the T-Jordan canonical form and its properties. The concept of T-minimal polynomial and T-characteristic polynomial are proposed. As a special case, we present properties when two tensors commutes based on the tensor T-product. We prove that the Cayley-Hamilton theorem also holds for tensor cases. Then we focus on the tensor decompositions:  T-polar, T-LU, T-QR and T-Schur decompositions of tensors are obtained. When a F-square tensor is not invertible with the T-product, we study the T-group inverse and T-Drazin inverse which can be viewed as the extension of matrix cases. The expression of T-group and T-Drazin inverse are given by the T-Jordan canonical form. The polynomial form of T-Drazin inverse is also proposed.  In the last part, we give the T-core-nilpotent decomposition and show that the T-index and T-Drazin inverse can be given by a limiting formula.
\bigskip

  \hspace{-14pt}{\bf Keywords.} T-Jordan canonical form, T-function, T-index, tensor decomposition, T-Drazin inverse, T-group inverse, T-core-nilpotent decomposition.

  \bigskip

  \hspace{-14pt}{\bf AMS Subject Classifications.} 15A48, 15A69, 65F10, 65H10, 65N22.

\end{abstract}

\newpage

\section{Introduction}
The Jordan canonical form of a linear operator or a square matrix has wide applications \cite{Horn1}. It gives a classification of matrices based on the similar relation which is the main equivalent relation in the matrix theory. However, this kind of equivalence  can not be extended to tensors because the multiplication between tensors are not well-defined. On the other hand, the Jordan canonical form is always used to define matrix functions and give the expressions of Drazin inverse for a  singular matrix. The group inverse \cite{Ben1, 2Campbell1} and the Drazin inverse \cite{2Drazin1, 2Wei1} are two  generalized inverses that people usually use and research in matrix theories.  The Drazin inverse is proved to be useful in Markov chain, matrix differential equations and matrix linear systems etc. \cite{2Campbell1}.
%In 1999, the Drazin inverse was extended to $C^*$-algebras by Koliha \cite{2Koliha1, %2Koliha2}.
\par
There are two important kinds of products between tensors, which is the tensor Einstein product and the tensor T-product. Both the tensor Moore-Penrose inverse and Drazin inverse have been established for Einstein product by Jin, Bai, Bentez and Liu \cite{2Jin1} and Sun, Zheng, Bu and Wei \cite{2Sun1}, respectively.

The perturbation theory for Moore-Penrose inverse of tensor via Einstein product was presented in \cite{Ma}. The outer inverse, core inverse and core-EP inverse of tensor based on the Einstein product have been investigated by Predrag et al. \cite{Sahoo,Predragouter}.
Recently, the Bernstein concentration inequality has also been proposed for tensors with the Einstein product by Luo, Qi and Toint \cite{2Luo1}.\par
On the other hand, the tensor T-product introduced by Kilmer \cite{Kilmer1} has been proved to be of great use in many areas, such as image processing \cite{Kilmer1, Martin1, Soltani1, Tarzanagh1}, computer vision \cite{Hao1}, signal processing \cite{Chan2, Liu1, Long1}, low rank tensor recovery and robust tensor PCA \cite{Kong2,Liu1},  data completion and denoising \cite{Hu2, Long1,Wang1}. An approach of linearization is provided by the T-product to transfer tensor multiplication to matrix multiplication by the discrete Fourier transformation and the theories of block circulant matrices \cite{Chan1, Jin1}. Due to the importance of the tensor T-product, Lund \cite{Lund1} gave the definition of tensor functions based on the T-product of third-order F-square tensors in her Ph.D thesis in  2018. The definition of T-function is given by
$$
f^{\Diamond}(\tens{A})=\fold(f(\bcirc(\tens{A}))\widehat{E_1}^{np\times n}),
$$
where `$\bcirc(\tens{A})$' is the block circulant matrix \cite{Jin1} generated by the F-square tensor $\tens{A}\in \mathbb{C}^{n\times n\times p}$. The T-function is also proved to be useful in stable tensor neural networks for rapid deep learning \cite{2Newman}. Special kinds of T-function such as tensor power has been used by Gleich, Chen and Varah \cite{Gleich1} in Arnoldi methods to compute the eigenvalues of tensors and diagonal tensor canonical form was also proposed by them. It is worth mention that, the tensor T-product and T-function is usually defined for third order tensors. For general cases, i.e., if we have a tensor $\mathbb{R}^{n \times n\times n_3\times \cdots\times n_m}$, we can combine the subscripts $n_3,n_4\ldots,n_m$ to one subscript $p$, then the problem will be changed to third order tensor cases \cite{Martin1}. \par

In this paper, we dedicate to research the F-square tensors and its properties. The organization is as follows. Firstly, the tensor T-similar relationship and the T-Jordan canonical form based on the T-product and their properties are introduced. We find the tensor T-Jordan canonical form is an upper-bi F-diagonal tensor whose off-F-diagonal entries are the linear combinations of $0$ and $1$ with coefficients $\omega^t$ divided by $p$, where $\omega={\textbf e}^{-2\pi {\textbf i}/p}$, ${\textbf i}=\sqrt{-1}, t=1,2,\ldots, p$, and $p$ is the number of frontal slices of the tensor. Based on the T-Jordan canonical form, we give the definition of F-Square tensor polynomial, tensor power series and their convergence radius. As an application, we present several power series of classical tensor T-functions and their convergence radius. Then we propose the T-minimal polynomial and T-characteristic polynomial which can be viewed as a special kind of standard tensor function. Cayley-Hamilton theorem also holds for tensors. As a special case, we discuss the properties when two tensors commutes with the tensor T-product. That is, when two tensors commutes, they can be F-diagonalized by the same invertible tensor. By induction, a family of commutative tensors also hold this property. We also find that normal tensors can be F-diagonalized by unitary tensors. Then we focus on the tensor decomposition. We give the definition of T-positive definite tensors. The T-polar, T-LU, T-QR and T-Schur decompositions of tensors are obtained.\par

Unfortunately, when  the original scalar function is the inverse function, the induced tensor T-function will not be well-defined to irreversible tensors. In the second part of the main results, we discuss the generalized inverse when a F-square tensor is not invertible. The T-group inverse and T-Drazin inverse which can be viewed as the extension of matrix cases are proposed, including their properties and constructions. We give the definition of tensor T-index. The relationship between the tensor T-index and the tensor T-minimal polynomial is obtained. The existence and uniqueness of T-group inverse are also proposed. The expression of T-group inverse according to the T-Jordan canonical form is also obtained. Then we focus on the T-Drazin inverse which can be viewed as the generalization of T-group inverse when the T-index of a tensor is known. Based on the T-Jordan canonical form, we obtain the expression of T-Drazin inverse. We find the T-Drazin inverse preserves similarity relationship between tensors.
%The same as the T-Moore-Penrose inverse \cite{2Miao1},
The T-Drazin inverse can be also given by the polynomial of the tensor. In the last part, we give the T-core-nilpotent decomposition of tensors and show that the tensor T-index and T-Drazin inverse can be given by a limited process.

\section{Preliminaries}
\subsection{Notation and index}
A new concept is proposed for multiplying third-order tensors, based on viewing a tensor as a stack of frontal slices. Suppose we have two tensors $\tens{A}\in \mathbb{R}^{m\times n \times p}$ and $\tens{B} \in \mathbb{R}^{n\times s \times p}$ and we denote their frontal faces respectively as $A^{(k)}\in \mathbb{R}^{m\times n}$ and $B^{(k)}\in \mathbb{R}^{n\times s}$, $k=1,2,\ldots, p$. The operations $\bcirc$, $\unfold$ and $\fold$ are defined as follows \cite{Hao1,Kilmer1, Kilmer2}:
$$
\bcirc(\tens{A}):=
\begin{bmatrix}
 A^{(1)} &  A^{(p)}  &  A^{(p-1)} & \cdots &  A^{(2)}\\

 A^{(2)} &  A^{(1)}  &  A^{(p)} & \cdots &  A^{(3)}\\

\vdots  & \ddots& \ddots & \ddots & \vdots\\

 A^{(p)} &  A^{(p-1)}  &  \ddots & A^{(2)} &  A^{(1)}\\
\end{bmatrix},\
\unfold (\tens{A}):=
\begin{bmatrix}
A^{(1)}\\

A^{(2)}\\

\vdots\\

 A^{(p)}\\
\end{bmatrix},
$$
and $\fold (\unfold(\tens{A})):=\tens{A}$, which means `$\fold$' is the inverse operator of $\unfold$. We can also define the corresponding inverse operation $\bcirc^{-1}: \mathbb{R}^{mp\times np}\rightarrow \mathbb{R}^{m\times n \times p}$ such that $\bcirc^{-1}(\bcirc({\tens{A}}))=\tens{A}$.\par

\subsection{The tensor T-Product}
The tensor T-product change problems to block circulant matrices which could be block diagonalizable by the fast Fourier transformation \cite{Chan1,Gleich1}. The calculation of T-product and T-SVD can be done fast  and stable because of the following reasons. First, the block circulant operator `bcirc' is only related to the structure of data, which can be constructed in a convenient way. Then the Fast Fourier Transformation and its inverse can be implemented stably  and efficiently. Its algorithm has been fully established. After transform the block circulant matrix into block diagonal matrix, functions can be done to each matrices.\par
On the other hand, the tensor T-product and T-functions does have applications in many scientific situations. For example, it can be used in conventional computed
tomography. Semerci, Hao, Kilmer and Miller \cite{Semerci1} introduced the tensor-based formulation and used the ADMM algorithm to solve the TNN model.
They give the quadratic approximation to the Poisson log-likelihood function for $k^{th}$ energy bin as a third order tensor, whose $k^{th}$ frontal slice is given by
$$
L_{k}({\textbf x}_{k})=({\textbf A}{\textbf x}_{k}-{\textbf m}_{k})^{\top}\Sigma_k^{-1}({\textbf A}{\textbf x}_{k}-{\textbf m}_{k}),\quad k=1,2,\ldots, p.
$$
where $\Sigma_k$ is treated as the weighting matrix. To minimize the objective function $L_{k}({\textbf x}_{k})$, it comes to solve the least squares problem, or equivalently, to obtain its T-generalized inverse, i.e., a special case of our generalized functions based on the T-product. They used the T-SVD and compute the T-least squares solution. Kilmer and Martin \cite{Kilmer2} also gave the definition of the standard inverse of tensors based on the T-product.\par

The following definitions and properties are introduced in \cite{Hao1,Kilmer1, Kilmer2}.
\begin{definition}\label{def2-1} {\rm (T-product)}
Let $\tens{A}\in \mathbb{R}^{m\times n \times p}$ and $\tens{B}\in \mathbb{R}^{n\times s \times p}$ be two real tensors. Then the T-product
$\tens{A}*\tens{B}$ is an $m\times s \times p$ real tensor defined by
$$
\tens{A}*\tens{B}:=\fold (\bcirc(\tens{A})\unfold(\tens{B})).
$$
\end{definition}
We introduce definitions of transpose, identity and orthogonal of tensors as follows.
\begin{definition}\label{def2-2} {\rm(Transpose and conjugate transpose)}
If $\tens{A}$ is a third order tensor of size $m\times n\times p $, then the transpose $\tens{A}^{\top}$ is obtained by transposing each of the frontal slices and then reversing the ordering of transposed frontal slices $2$ through $n$. The conjugate transpose $\tens{A}^{H}$ is obtained by conjugate transposing each of the frontal slices and then reversing the order of transposed frontal slices $2$ through $n$.
\end{definition}
\begin{definition}\label{def2-3} {\rm (Identity tensor)}
The $n\times n \times p $ identity tensor $\tens{I}_{nnp}$ is the tensor whose first frontal slice is the $n\times n$ identity matrix, and whose other frontal slices are all zeros.
\end{definition}
It is easy to check that $\tens{A}*\tens{I}_{nnp}=\tens{I}_{mmp}*\tens{A}=\tens{A}$ for $\tens{A}\in \mathbb{R}^{m\times n\times p}$.
\begin{definition}\label{def2-4} {\rm (Orthogonal and unitary tensor)}
An $n\times n\times p$ real-valued tensor $\tens{P}$ is orthogonal if $\tens{P}^{\top}*\tens{P}=\tens{P}*\tens{P}^{\top}=\tens{I}$.  An $n\times n\times p$ complex-valued tensor $\tens{Q}$ is unitary if $\tens{Q}^{H}*\tens{Q}=\tens{Q}*\tens{Q}^{H}=\tens{I}$.
\end{definition}
For a frontal square tensor $\tens{A}$ of size $n\times n \times p$, it has inverse tensor $\tens{B}$ $(=\tens{A}^{-1})$, provided that
$$
\tens{A}*\tens{B}=\tens{I}_{nnp}\ \ and \ \ \tens{B}*\tens{A}=\tens{I}_{nnp}.
$$
It should be noticed that invertible third order tensors of size $n\times n\times p$ form a group, since the invertibility of tensor $\ten{A}$ is equivalent to the invertibility of the matrix $\bcirc(\tens{A})$, and the set of invertible matrices form a group. The orthogonal tensors based on the tensor T-product also forms a group, since $\bcirc(\tens{Q})$ is an orthogonal matrix. \par
The concepts of T-range space, T-null space, tensor norm, and T-Moore-Penrose inverse are given as follows \cite{2Miao1}.
\begin{definition}\label{def2-5}
Let $\tens{A}$ be an $n_1\times n_2 \times n_3$ real-valued tensor.\\
{\rm (1)} \ The T-range space of $\tens{A}$, $\tens{R}(\tens{A}):= {\rm Ran}((F_p^H\otimes I_{n_1})\bcirc(\tens{A})(F_p\otimes I_{n_1}))$, `${\rm Ran}$' means the range space,\\
{\rm (2)} \ The T-null space of $\tens{A}$, $\tens{N}(\tens{A}):={\rm Null}((F_p^H\otimes I_{n_1})\bcirc(\tens{A})(F_p\otimes I_{n_1}))$, `${\rm Null}$' represents the null space,\\
{\rm (3)} \ The tensor unitary invariant norm $\norm{\tens{A}}:=\norm{\bcirc(\tens{A})}$, where the matrix norm $\norm{\cdot}$ should also be chosen as a unitary invariant norm.\\
{\rm (4)} \ The T-Moore-Penrose inverse \cite{2Miao1} $\tens{A}^{\dag}=\bcirc^{-1}((\bcirc(\tens{A}))^{\dag})$.
\end{definition}
\subsection{Tensor T-Function}
In this section, we recall for the functions of square matrices based on the Jordan canonical form \cite{Golub1,Higham1}. \par
Let $A\in \mathbb{C}^{n\times n}$ be a matrix with spectrum $\lambda(A):=\{\lambda_j\}_{j=1}^N$, where $N\leq n$ and $\lambda_j$ are distinct. Each $m\times m$ Jordan block $J_m(\lambda)$ of an eigenvalue $\lambda$ has the form
$$
J_m(\lambda)=
\begin{bmatrix}
\lambda & 1 &  &  \\

  & \lambda & \ddots  & \\

  & & \ddots &1\\

 & &  & \lambda
\end{bmatrix}\in \mathbb{C}^{m\times m}.
$$
Suppose that $A$ has the Jordan canonical form
$$
A=XJX^{-1}=X {\rm diag}(J_{m_1}(\lambda_{j_1}),\cdots, J_{m_p}(\lambda_{j_p}))X^{-1},
$$
with $p$ blocks of sizes $m_i$ such that $\sum_{i=1}^{p}m_i=n$, and the eigenvalues $\{\lambda_{j_k}\}_{k=1}^{p}\in {\rm spec}(A)$. \par
\begin{definition}\label{def2-6} {\rm (Matrix function)}
Suppose that $A\in \mathbb{C}^{n\times n}$ has the Jordan canonical form and the matrix function is defined as
$$
f(A):=X f(J)X^{-1},
$$
where $f(J):={\rm diag}(f(J_{m_1}(\lambda_{j_1})),\cdots, f(J_{m_p}(\lambda_{j_p})))$, and
$$
f(J_{m_i}(\lambda_{j_i})):=
\begin{bmatrix}
f(\lambda_{j_k}) & f'(\lambda_{j_k}) & \frac{f''(\lambda_{j_k})}{2!} & \cdots & \frac{f^{(n_{j_k}-1)}(\lambda_{j_k})}{(n_{j_k}-1)!} \\

0 & f(\lambda_{j_k}) & f'(\lambda_{j_k}) &  \cdots & \vdots \\

\vdots & \ddots & \ddots &\ddots & \frac{f''(\lambda_{j_k})}{2!} \\

\vdots &   & \ddots & \ddots& f'(\lambda_{j_k}) \\

0 & \cdots & \cdots & 0 & f(\lambda_{j_k})\\
\end{bmatrix}\in \mathbb{C}^{m_i\times m_i}.
$$
\end{definition}
There are various properties of matrix functions throughout matrix analysis. Here we give some of these properties and the proofs that could be found in the monograph \cite{Higham1}.
\begin{lemma}\label{lem2-1}
Assume that $A$ is a complex matrix of size $n \times n$ and $f$ is a function defined on the spectrum of $A$. Then we have \\
${\rm (1)}\  f(A)A=Af(A)$,\\
${\rm (2)}\  f(A^{H})=f(A)^{H}$,\\
${\rm (3)}\  f(XAX^{-1})=Xf(A)X^{-1}$,\\
${\rm (4)}\  f(\lambda)\in {\rm spec}(f(A))$ for all $\lambda \in {\rm spec}(A) $, where `${\rm spec}$' means the spectrum of a matrix.
\end{lemma}
By using the concept of T-product, the matrix function can be generalized to tensors of size $n\times n \times p$. Suppose we have tensors $\tens{A}\in \mathbb{C}^{n\times n\times p}$ and $\tens{B}\in \mathbb{C}^{n\times s\times p}$, then the tensor T-function of $\tens{A}$ is defined by \cite{Lund1}:
$$
f(\tens{A})*\tens{B}:=\fold(f(\bcirc(\tens{A}))\cdot \unfold(\tens{B})),
$$
or equivalently
$$
f(\tens{A}):=\fold(f(\bcirc(\tens{A}))\widehat{E_{1}}^{np\times n}),
$$
where $\widehat{E_{1}}^{np\times n}=\hat{e}_k^p \otimes I_{n\times n}$, $\hat{e}_k^p\in \mathbb{C}^p$ is the vector of all zeros except for the $k$th entry and $I_{n\times n}$ is the identity matrix, `$\otimes$' is the matrix Kronecker product \cite{Horn1}. \par

There is another way to express $\widehat{E_{1}}^{np\times n}$:
$$
\widehat{E_{1}}^{np\times n}=
\begin{bmatrix}
I_{n\times n}\\
0\\
\vdots\\
0
\end{bmatrix}
=
\begin{bmatrix}
1\\
0\\
\vdots\\
0
\end{bmatrix}\otimes I_{n\times n }=\unfold (\tens{I}_{n\times n \times p}).
$$
Note that $f$ on the right-hand side of the equation is merely the matrix function defined above, so the tensor T-function is well-defined.\par
From this definition, we can see that for a tensor $\mathbb{\tens{A}}\in \mathbb{C}^{n\times n\times p}$, $\bcirc(\tens{A})$ is a block circulant matrix of size $np\times np$. The frontal faces of $\tens{A}$ are the block entries of $A\widehat{E_{1}}^{np\times n}$, then $\tens{A}=\fold(A\widehat{E_{1}}^{np\times n})$, where $A=\unfold(\tens{A})$.\par
In order to get further properties of generalized tensor functions, we make some review of the results on block circulant matrices and the tensor T-product.
\begin{lemma}\label{lem2-2} {\rm\cite{Chan1}}
Suppose $A, B\in \mathbb{C}^{np\times np}$ are block circulant matrices with $n\times n$ blocks. Let $\{\alpha_j\}_{j=1}^k$ be scalars. Then $A^{\top}$, $A^{H}$, $\alpha_1 A+\alpha_2 B$, $AB$, $q(A)=\sum_{j=1}^k \alpha_j A^j$ and $A^{-1}$ are also block circulant matrices.
\end{lemma}
\begin{lemma}\label{lem2-3} {\rm \cite{Lund1}}
   Let tensors $\tens{A}\in \mathbb{C}^{n\times n\times p}$ and $\tens{B}\in \mathbb{C}^{n\times s\times p}$. Then \\
{\rm (1)}\ $\unfold (\tens{A})=\bcirc(\tens{A})\widehat{E_{1}}^{np\times n}$,\\
{\rm (2)}\ $\bcirc(\fold(\bcirc(\tens{A})\widehat{E_{1}}^{np\times n}))=\bcirc(\tens{A})$,\\
{\rm (3)}\ $\bcirc(\tens{A}*\tens{B})=\bcirc(\tens{A})\bcirc(\tens{B})$,\\
{\rm (4)}\ $\bcirc(\tens{A})^j=\bcirc(\tens{A}^j)$, for all $j=0,1,\ldots$,\\
{\rm (5)}\ $(\tens{A}*\tens{B})^{H}=\tens{B}^H *\tens{A}^H$,\\
{\rm (6)}\ $\bcirc(\tens{A}^{\top})=(\bcirc(\tens{A}))^{\top}$,\ $\bcirc(\tens{A}^{H})=(\bcirc(\tens{A}))^{H}$.
\end{lemma}

\section{Main Results}

\subsection{T-Jordan canonical form}
It is a well-known result that every matrix has its Jordan canonical form. The Jordan canonical form is named after Camille Jordan, who first introduced the Jordan decomposition theorem in 1870 and it is of great use in differential equations, interpolation theory, operation theory, functional analysis, matrix computations  \cite{Golub1, Higham1, Horn1}.\par
For third order tensors, we can also introduce the Jordan canonical form based on the tensor T-product. Similar to matrix cases, we will only consider complex F-square third order tensors, that is, tensors in the space $\mathbb{C}^{n\times n\times p}$.
\begin{definition}\label{def2-7} {\rm (Similar transformation)} Suppose $\tens{A}$, $\tens{B}\in\mathbb{C}^{n\times n\times p}$ are two F-square complex tensors. We say $\tens{B}$ is similar to $\tens{A}$ if there exists a invertible tensor $\tens{P}\in\mathbb{C}^{n\times n\times p}$ satisfying
\begin{equation}
\tens{B}=\tens{P}^{-1}*\tens{A}*\tens{P}.
\end{equation}
\end{definition}
Now we dedicate to find the canonical form under the similar relation. We have the following important lemma.
\begin{lemma}\label{lem2-4}
Suppose $A^{(1)}, A^{(2)}, \cdots, A^{(p)}, B^{(1)}, B^{(2)}, \cdots, B^{(p)}\in \mathbb{C}^{n\times n}$ are complex matrices satisfying
$$
\begin{bmatrix}
 A^{(1)} &  A^{(p)}  &  A^{(p-1)} & \cdots &  A^{(2)}\\

 A^{(2)} &  A^{(1)}  &  A^{(p)} & \cdots &  A^{(3)}\\

\vdots  & \ddots& \ddots & \ddots & \vdots\\

 A^{(p)} &  A^{(p-1)}  &  \ddots & A^{(2)} &  A^{(1)}\\
\end{bmatrix}=
(F_{p}\otimes I_{n} )
\begin{bmatrix}
B^{(1)} &  &  &  \\

  & B^{(2)} &   & \\

  & & \ddots &\\

 & &  & B^{(p)}
\end{bmatrix}
(F_{p}^{H}\otimes I_{n} ),
$$
where $F_{p}$ is the discrete Fourier matrix of size $p\times p$.
Then $B^{(1)}, B^{(2)}, \cdots, B^{(p)}$ are diagonal (sub-diagonal, upper-triangular, lower-triangular) matrices if and only if $A^{(1)}, A^{(2)}, \cdots, A^{(p)}$ are diagonal (sub-diagonal, upper-triangular, lower-triangular) matrices.
\end{lemma}
\begin{proof}
By the definition of the Fourier matrix and matrix multiplication, each subblock $A^{(i)}\ ( i= 1,2,\ldots, p)$ are the linear combination in the complex field of the subblocks $B^{(1)}, B^{(2)}, \cdots, B^{(p)}$.\par
Furthermore, we reveal the relationship between complex matrices $A^{(1)}, A^{(2)}, \cdots, A^{(p)}$ and $B^{(1)}, B^{(2)}, \cdots, B^{(p)}$:
\begin{equation}
\left\{
\begin{aligned}
&A^{(1)}=\frac{1}{p}(\omega^0 B^{(1)}+\omega^0 B^{(2)}+\cdots+\omega^0 B^{(p)}),\\
&A^{(2)}=\frac{1}{p}(\omega^0 B^{(1)}+\omega^1 B^{(2)}+\cdots+\omega^{p-1} B^{(p)}),\\
&\ \ \ \ \ \ \ \ \ \ \ \ \ \ \ \ \ \ \ \ \ \ \ \ \ \ \cdots\\
&A^{(p)}=\frac{1}{p}(\omega^0 B^{(1)}+\omega^{p-1} B^{(2)}+\cdots+\omega^{(p-1)(p-1)} B^{(p)}),\\
\end{aligned}
\right.
\end{equation}
where $\omega=e^{-2\pi {\textbf i}/p}$ is the primitive $p$-th root of unity and is usually called the {\it phase term}.\par
On the other hand, since the matrix
$$
\begin{bmatrix}
\omega^0&\omega^0&\cdots&\omega^0\\
\omega^0&\omega^1&\cdots&\omega^{p-1} \\
&&\ddots&\\
\omega^0&\omega^{p-1}&\cdots&\omega^{(p-1)(p-1)}
\end{bmatrix}=
\begin{bmatrix}
1&1&\cdots&1\\
1&\omega^1&\cdots&\omega^{p-1} \\
&&\ddots&\\
1&\omega^{p-1}&\cdots&\omega^{(p-1)(p-1)}
\end{bmatrix}
$$
is the Vandermonde matrix \cite{Horn1}, which is invertible, matrices $B^{(1)}, B^{(2)}, \cdots, B^{(p)}$ are also the linear combination of $A^{(1)}, A^{(2)}, \cdots, A^{(p)}$.
\end{proof}
 \begin{figure}
  \centering
  \includegraphics[width=1\textwidth]{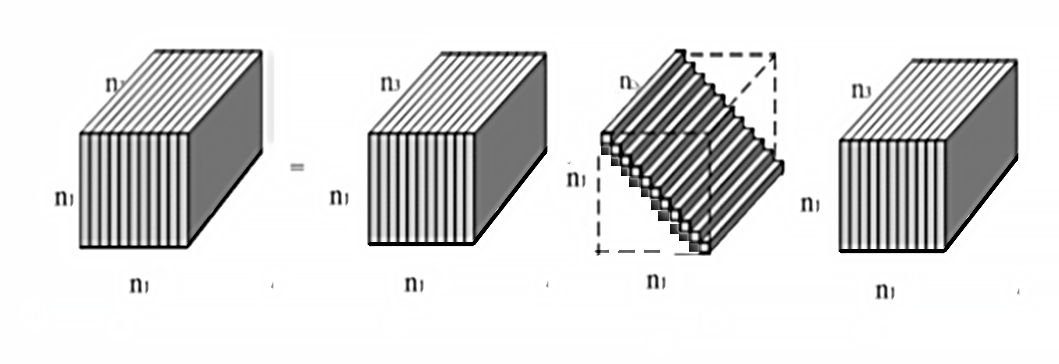}
  \caption{T-Jordan Canonical Form of Tensors}
\end{figure}
\begin{theorem}\label{the2-1} {\rm (T-Jordan Canonical Form)} Let $\tens{A}\in \mathbb{C}^{n\times n\times p}$ be a complex tensor, then there exists an invertible tensor $\tens{P}\in \mathbb{C}^{n\times n\times p}$ and a F-upper-bi-diagonal tensor $\tens{J}\in \mathbb{C}^{n\times n\times p}$ such that
$$
\tens{A}=\tens{P}^{-1}*\tens{J}*\tens{P}.
$$
\end{theorem}
\begin{proof}
From Lemma \ref{lem2-4}, for the complex tensor $\tens{A}$, we have
$$
\begin{aligned}
\bcirc(\tens{A})&=\begin{bmatrix}
 A^{(1)} &  A^{(p)}  &  A^{(p-1)} & \cdots &  A^{(2)}\\

 A^{(2)} &  A^{(1)}  &  A^{(p)} & \cdots &  A^{(3)}\\

\vdots  & \ddots& \ddots & \ddots & \vdots\\

 A^{(p)} &  A^{(p-1)}  &  \ddots & A^{(2)} &  A^{(1)}\\
\end{bmatrix}\\
&=(F_{p}\otimes I_{n} )
\begin{bmatrix}
B^{(1)} &  &  &  \\

  & B^{(2)} &   & \\

  & & \ddots &\\

 & &  & B^{(p)}
\end{bmatrix}
(F_{p}^{H}\otimes I_{n} ).
\end{aligned}
$$
Since each matrix $B^{(i)}$ ($i=1,2, \ldots, p$) has its Jordan canonical form $B^{(i)}=(P^{(i)})^{-1} C^{(i)}P^{(i)}$, $P^{(i)}$, $C^{(i)}\in\mathbb{C}^{n\times n}$, $i=1,2, \ldots, p$,  it turns out
$$
\begin{aligned}
\bcirc(\tens{A})
&=(F_{p}\otimes I_{n} )
\begin{bmatrix}
(P^{(1)})^{-1} &  &  &  \\

  & (P^{(2)})^{-1} &   & \\

  & & \ddots &\\

 & &  & (P^{(p)})^{-1}
\end{bmatrix}
(F_{p}^{H}\otimes I_{n} )\\
&\times
(F_{p}\otimes I_{n} )
\begin{bmatrix}
C^{(1)} &  &  &  \\

  & C^{(2)} &   & \\

  & & \ddots &\\

 & &  & C^{(p)}
\end{bmatrix}
(F_{p}^{H}\otimes I_{n} )\\
&\times (F_{p}\otimes I_{n} )
\begin{bmatrix}
P^{(1)} &  &  &  \\

  & P^{(2)} &   & \\

  & & \ddots &\\

 & &  & P^{(p)}
\end{bmatrix}
(F_{p}^{H}\otimes I_{n} ).\\
\end{aligned}
$$
We have
$$
\begin{aligned}
\bcirc(\tens{A})
&=\left((F_{p}\otimes I_{n} )
\begin{bmatrix}
P^{(1)} &  &  &  \\

  & P^{(2)} &   & \\

  & & \ddots &\\

 & &  & P^{(p)}
\end{bmatrix}
(F_{p}^{H}\otimes I_{n} )\right)^{-1}\\
&\times
(F_{p}\otimes I_{n} )
\begin{bmatrix}
C^{(1)} &  &  &  \\

  & C^{(2)} &   & \\

  & & \ddots &\\

 & &  & C^{(p)}
\end{bmatrix}
(F_{p}^{H}\otimes I_{n} )\\
&\times (F_{p}\otimes I_{n} )
\begin{bmatrix}
P^{(1)} &  &  &  \\

  & P^{(2)} &   & \\

  & & \ddots &\\

 & &  & P^{(p)}
\end{bmatrix}
(F_{p}^{H}\otimes I_{n} )\\
&=\bcirc(\tens{P})^{-1}\bcirc(\tens{J})\bcirc(\tens{P}),
\end{aligned}
$$
which is equivalent to
$$
\tens{A}=\tens{P}^{-1}*\tens{J}*\tens{P}.
$$
We call the diagonal elements of $C^{(i)}$ {\textbf{T-eigenvalues}} of $\tens{A}$. \par
Furthermore, if we suppose that
$$
(F_{p}\otimes I_{n} )\times
\begin{bmatrix}
C^{(1)} &  &  &  \\

  & C^{(2)} &   & \\

  & & \ddots &\\

 & &  & C^{(p)}
\end{bmatrix}
\times (F_{p}^H\otimes I_{n} )=
\begin{bmatrix}
 J^{(1)} &  J^{(p)}  &  J^{(p-1)} & \cdots &  J^{(2)}\\

J^{(2)} &  J^{(1)}  &  J^{(p)} & \cdots &  J^{(3)}\\

\vdots  & \ddots& \ddots & \ddots & \vdots\\

J^{(p)} &  J^{(p-1)}  &  \ddots & J^{(2)} &  J^{(1)}\\
\end{bmatrix},
$$
where $J^{(i)}$ ($i=1,2, \ldots, p$) is the $i$-th frontal slice of the tensor $\tens{J}$.

 By Lemma \ref{lem2-4}, we obtain that each $J^{(i)}$ is a complex upper two-diagonal matrix and is the linear combination with phase terms of $C^{(i)}$:
\begin{equation}
\left\{
\begin{aligned}
&J^{(1)}=\frac{1}{p}(\omega^0 C^{(1)}+\omega^0 C^{(2)}+\cdots+\omega^0 C^{(p)}),\\
&J^{(2)}=\frac{1}{p}(\omega^0 C^{(1)}+\omega^1 C^{(2)}+\cdots+\omega^{p-1} C^{(p)}),\\
&\ \ \ \ \ \ \ \ \ \ \ \ \ \ \ \ \ \ \ \ \ \ \ \ \ \ \cdots\\
&J^{(p)}=\frac{1}{p}(\omega^0 C^{(1)}+\omega^{p-1} C^{(2)}+\cdots+\omega^{(p-1)(p-1)} C^{(p)}).\\
\end{aligned}
\right.
\end{equation}
 Furthermore, since the above equations and operations are all invertible,  the Jordan canonical form $\tens{J}$ of the tensor $\tens{A}$ is unique.
\end{proof}
\begin{remark}\label{rem2-1} Now we establish the T-Jordan canonical form. If a tensor $\tens{A}$ is a real tensor, then its Jordan canonical form $\tens{J}$ will also be a real tensor due to the discrete Fourier transformation.
\end{remark}
The definition of T-function defined by Lund \cite{Lund1} has the equivalent expression as follows:
\begin{theorem}\label{the2-2}
If the tensor $\tens{A}$ has the factorization,
$$
\bcirc(\tens{A})=(F_p^H\otimes I_n)
\begin{bmatrix}
A_1&&&\\
&A_2&&\\
&&\ddots&\\
&&&A_p
\end{bmatrix}(F_p\otimes I_n),
$$
then the above definition of standard tensor T-function is equivalent to
\begin{equation}
\begin{aligned}
f(\tens{A})&=\bcirc^{-1}(f(\bcirc(\tens{A})))\\
&=\bcirc^{-1}\left((F_p^H\otimes I_n)
\begin{bmatrix}
f(A_1)&&&\\
&f(A_2)&&\\
&&\ddots&\\
&&&f(A_p)
\end{bmatrix}(F_p\otimes I_n)\right).
\end{aligned}
\end{equation}
\end{theorem}
\begin{proof}
By the definition of tensor T-product, we have
$$
f(\tens{A})*\tens{B}=\fold(\bcirc(f(\tens{A}))\unfold(\tens{B})).
$$
By taking $\tens{B}=\tens{I}$, we have
$$
f(\tens{A})=\fold(\bcirc(f(\tens{A}))\unfold(\tens{B}))=\fold(\bcirc(f(\tens{A}))\widehat{E_{1}}^{np\times n}),
$$
by comparing with the definition of tensor T-function $f(\tens{A})= {\rm fold}(f(\bcirc(\tens{A}))\widehat{E_{1}}^{np\times n})$, we have
$$
\bcirc(f(\tens{A}))=f(\bcirc(\tens{A})),
$$
which comes to the result.
\end{proof}
The following properties hold for the standard tensor function.
\begin{lemma}\label{lem2-5} {\rm \cite{Lund1}} Let $\ten{A}\in \mathbb{C}^{n\times n\times p}$ and $f:\mathbb{C}\rightarrow \mathbb{C}$ be defined on a region in the complex plane containing all the spectrum of the matrix $\bcirc(\ten{A})$. Then the standard tensor function $f(\tens{A})$ satisfies the following properties:\\
{\rm (1)} $f(\tens{A})$ commutes with $\tens{A}$,\\
{\rm (2)} $f(\tens{A}^H)=f(\tens{A})^H$,\\
{\rm (3)} $f(\tens{X}*\tens{A}*\tens{X}^{-1})=\tens{X}*f(\tens{A})*\tens{X}^{-1}$,\\
{\rm (4)} If $f(x)=\sum_{i=0}^n c_i x^i$ is a polynomial scalar function of degree $n$, $c_i\in \mathbb{C}, (i=0,1,\ldots, n$). Then the standard tensor function satisfies $f(\tens{A})=\sum_{i=0}^n c_i \tens{A}^i$, where $\tens{A}^i=\underbrace{\tens{A}*\tens{A}*\cdots*\tens{A}}_{i\  times}$.
\end{lemma}
\begin{corollary}\label{cor2-1}
Specially, if a tensor $\tens{A}\in \mathbb{C}^{n\times n\times p}$ has the Jordan canonical form $\tens{A}=\tens{P}^{-1}*\tens{J}*\tens{P}$, $f:\mathbb{C}^{n \times n \times p}\rightarrow \mathbb{C}^{n\times n\times p}$ is a standard tensor function, then
\begin{equation}
f(\tens{A})=\tens{P}^{-1}*f(\tens{J})*\tens{P}.
\end{equation}
\end{corollary}
The concept of nilpotent matrix \cite{Horn1} can be extended to tensors as follows.

\begin{definition}\label{def2-8} {\rm (Nilpotent tensor)} A tensor $\tens{A}\in \mathbb{C}^{n\times n\times p}$ is called nilpotent, if there exists a positive integer $s\in \mathbb{Z}$ such that $\tens{A}^s=\tens{O}$. If $s\in \mathbb{Z}$ is the smallest number satisfies the equation $\tens{A}^s=\tens{O}$, then we call $s$ the nilpotent index of $\tens{A}$.
\end{definition}

By the definition of T-eigenvalues, we have the following corollary.
\begin{corollary}\label{cor2-2} The tensor $\tens{A}$ is a nilpotent tensor if and only if all the T-eigenvalues of $\tens{A}$ equal to $0$.
\end{corollary}
\subsection{F-square tensor power series}
As an direct application of T-Jordan canonical form, we discuss the F-square tensor power series as an extension of Lund's results \cite{Lund1} and Theorem \ref{the2-2}.\par
\begin{definition}\label{def2-9} {\rm (F-square tensor polynomial)} Let $\tens{A}\in\mathbb{C}^{n\times n\times p}$ be a complex tensor. We call
$$
p(\tens{A})=\sum_{k=0}^{m}a_k\tens{A}^k=a_m\tens{A}^m+a_{m-1}\tens{A}^{m-1}+\cdots+a_1\tens{A}+a_0\tens{I}_{nnp}
$$
the polynomial of $\tens{A}$, whose degree is $m$.
\end{definition}
\begin{definition}\label{def2-10} {\rm (F-square tensor power series)} Let $\tens{A}\in\mathbb{C}^{n\times n\times p}$ be a complex tensor. We call
$$
p(\tens{A})=\sum_{k=0}^{\infty}a_k\tens{A}^k=a_0\tens{I}_{nnp}+a_1\tens{A}+a_2\tens{A}^2+\cdots
$$
the power series of $\tens{A}$.
\end{definition}
By the $(3)$ of Lemma \ref{lem2-5}, we can transfer the properties of the power series of tensor $\tens{A}$ to the properties of the power series. By Theorem \ref{the2-2}, we directly give the convergence theorem of tensor series as follows.
\begin{theorem}\label{the2-3} {\rm (Tensor series convergence)} Let $\tens{A}\in \mathbb{C}^{n\times n\times p}$ be a complex tensor and $f:\mathbb{C}\rightarrow \mathbb{C}$ be a complex power series given by
$$
f(z)=\sum_{k=0}^{\infty} a_k z^k.
$$
The convergence radius of $f$ is denoted by $\rho\in \mathbb{R}$, and $\rho$ is also called the convergence radius of tensor series $\sum\limits_{k=0}^{\infty}a_k\tens{A}^k$. Suppose the maximum of the T-eigenvalues of tensor $\tens{A}$ is $\lambda_0=\max\{|\lambda^{(i)}_j|, i=1,2,\ldots,p, j=1,2,\ldots,n\}$. If
$$
\rho >\lambda_0,
$$
then the tensor power series converges and the T-eigenvalues of $\sum\limits_{k=0}^{\infty}a_k\tens{A}^k$ is
$$
\sum_{k=0}^{\infty}a_k (\lambda^{(i)}_j)^k, \quad i=1,2,\ldots,p,\ j=1,2,\ldots,n.
$$
\end{theorem}
By using Theorem \ref{the2-3}, we can extend some special kinds of scalar power series to tensor power series as follows,
$$
{\rm exp}(\tens{A})={\textbf e}^{\tens{A}}=\sum_{k=0}^{\infty}\frac{1}{k!}\tens{A}^k,\quad \sin(\tens{A})=\sum_{k=0}^{\infty}\frac{(-1)^{k-1}}{(2k-1)!}\tens{A}^{2k-1},
$$
$$
\cos(\tens{A})=\sum_{k=0}^{\infty}\frac{(-1)^{k}}{(2k)!}\tens{A}^{2k},\quad \ln(\tens{I}+\tens{A})=\sum_{k=0}^{\infty}\frac{(-1)^k}{k+1}\tens{A}^{k+1},
$$
$$
(\tens{I}+\tens{A})^{\alpha}=\sum_{k=0}^{\infty} \binom{\alpha}{k}\tens{A}^k, \quad \binom{\alpha}{k}=\frac{\alpha(\alpha-1)\cdots(\alpha-k+1)}{k!},\ k=0,1,\ldots, \ \alpha \in \mathbb{R}.
$$

The standard tensor function $\exp(\tens{A})$, $\sin(\tens{A})$, $\cos(\tens{A})$ can be defined to all F-square tensors since the convergence radius of complex valued power series
$$
{\textbf e}^z=\sum_{k=0}^{\infty}\frac{1}{k!}z^k,\quad \sin(z)=\sum_{k=0}^{\infty}\frac{(-1)^{k-1}}{(2k-1)!}z^{2k-1},\quad \cos(z)=\sum_{k=0}^{\infty}\frac{(-1)^{k}}{(2k)!}z^{2k}
$$
are infinity. However, the convergence radius of complex valued power series
$$
{\rm ln}(1+z)=\sum_{k=0}^{\infty}\frac{(-1)^k}{k}z^k, \quad (1+z)^{\alpha}=\sum_{k=0}^{\infty} \binom{\alpha}{k}z^k
$$
are $1$, the corresponding tensor series $\ln(\tens{I}+\tens{A})$ and $(\tens{I}+\tens{A})^{\alpha}$ can only be defined for F-square tensors whose modules of T-eigenvalues are less than $1$.\par
Moreover, for the above tensor functions, we have the relation
$$
\cos^2(\tens{A})+\sin^2(\tens{A})=\tens{I}
$$
for all complex tensors $\tens{A}$ and
$$
\exp(\ln(\tens{I}+\tens{A}))=\tens{I}+\tens{A}
$$
for tensors whose modules of T-eigenvalues are less than $1$.
\begin{corollary}\label{cor2-3} Let $\tens{A},\tens{B}\in\mathbb{C}^{n\times n\times p}$ be two complex tensors commute with each other. Then
$$
\exp(\tens{A})*\exp(\tens{B})=\exp(\tens{B})*\exp(\tens{A})=\exp(\tens{A}+\tens{B}).
$$
\end{corollary}
\begin{corollary}\label{cor2-4} Let $\tens{A}\in\mathbb{C}^{n\times n\times p}$ be a complex tensor. Then there exists a complex tensor $\tens{B}\in\mathbb{C}^{n\times n\times p}$, such that
$$
\tens{A}=\exp (\tens{B}),
$$
when $\tens{A}$ is invertible, the tensor equation
$$
\exp(\tens{X})=\tens{A}
$$
has a solution.
\end{corollary}
It should be noticed that the solution of tensor equation $\exp(\tens{X})=\tens{A}$ is not unique. In fact since
$$
{\textbf e}^{2k\pi {\textbf i} }=1,\quad {\textbf i}=\sqrt{-1}, k=0,\pm1,\pm2,\cdots,
$$
so if $\exp(\tens{B})=\tens{A}$, then
$$
\exp(\tens{B}+2k\pi {\textbf i} \tens{I})=\tens{A},\quad k=0,\pm1,\pm2,\cdots.
$$
This theorem also tells us that every third order tensor has its $\alpha$-th root \cite{Higham1}.

\begin{corollary}\label{cor2-5} Let $\tens{A}\in\mathbb{C}^{n\times n\times p}$ be a complex invertible tensor, $\alpha\in \mathbb{C}$ be a non-zero complex number. Then there exists a complex tensor $\tens{B}$ such that
$$
\tens{B}^{\alpha}=\tens{A}.
$$
That is the tensor equation
$$
\tens{X}^{\alpha}=\tens{A}
$$
has a solution.
\end{corollary}
\begin{proof}
There exists a complex tensor $\tens{C}$, such that $\tens{A}=\exp(\tens{C})$. Let
$$
\tens{B}=\exp\left(\frac{1}{\alpha}\tens{C}\right).
$$
Then we have $\tens{B}^{\alpha}=\left(\exp\left(\frac{1}{\alpha}\tens{C}\right)\right)^{\alpha}=
\exp(\tens{C})=\tens{A}$.
\end{proof}
Also, the $\alpha$-th root of $\tens{A}$ is not unique. In fact, if $\tens{B}^{\alpha}=\tens{A}$, then
$$
\left(\tens{B}*\exp\left(\frac{2k\pi {\textbf i} }{\alpha}\right)\right)^{\alpha}=\tens{B}^{\alpha}*\exp(2k\pi {\textbf i})=\tens{B}^{\alpha}=\tens{A}.
$$
For irreversible tensors, generally, they may not have its $\alpha$-th root.

\subsection{Commutative tensor family}
The set of diagonalizable matrices is an important class of matrices in the linear algebra. In this subsection, we talks about a special class of tensor, which is the F-diagonalizable tensor.
\begin{definition}\label{def2-11} {\rm \cite{Lund1}} {\rm (F-diagonalizable tensor)} We call a tensor $\tens{A}\in \mathbb{C}^{n\times n\times p}$, which has the Jordan decomposition $\tens{A}=\tens{P}^{-1}*\tens{J}*\tens{P}$, is a F-diagonalizable tensor if its Jordan canonical form $\tens{J}\in \mathbb{C}^{n\times n\times p}$ is a F-diagonal tensor, i.e. all the frontal slices of $\tens{J}$ are diagonal matrices.
\end{definition}
\begin{definition}\label{def2-12} {\rm (T-Commutative)} Let $\tens{A}$, $\tens{B}\in\mathbb{C}^{n\times n\times p}$ be two tensors. We call tensor $\tens{A}$ commutes with $\tens{B}$  if
\begin{equation}
[\tens{A},\tens{B}]:=\tens{A}*\tens{B}-\tens{B}*\tens{A}=\tens{O}.
\end{equation}
\end{definition}
For matrix cases, we have the following lemma.
\begin{lemma}\label{lem2-6} {\rm \cite{Horn1}}
Let $A$, $B\in\mathbb{C}^{n\times n}$ be two diagonal matrices. If $A$ commutes with $B$ and $A$ is diagonalizable by the matrix $P$,
$$
A=P^{-1}DP,
$$
where $D\in\mathbb{C}^{n\times n}$ is a diagonal tensor, then $B$ can also be diagonalized by matrix $P$, that is
$$
B=P^{-1}D'P,
$$
where $D'\in\mathbb{C}^{n\times n}$ is also a diagonal matrix.
\end{lemma}
If tensors are commutative with T-product and both can be F-diagonalized, then they can be F-diagonalized by the same invertible tensor $\tens{P}$. In order to prove this result, we shall use the above Lemma \ref{lem2-4} and \ref{lem2-6}.
\begin{theorem}\label{the2-6} {\rm (Diagonalizable simultaneously)} Let $\tens{A}$, $\tens{B}\in\mathbb{C}^{n\times n\times p}$ be two F-diagonalizable tensors. If $\tens{A}$ commutes with $\tens{B}$ and $\tens{A}$ is F-diagonalizable by tensor $\tens{P}\in\mathbb{C}^{n\times n\times p}$,
$$
\tens{A}=\tens{P}^{-1}*\tens{D}*\tens{P},
$$
where $\tens{D}\in\mathbb{C}^{n\times n\times p}$ is a F-diagonal tensor, then $\tens{B}$ can also be diagonalized by tensor $\tens{P}$, that is
$$
\tens{B}=\tens{P}^{-1}*\tens{D}'*\tens{P},
$$
where $\tens{D}'\in\mathbb{C}^{n\times n\times p}$ is also a F-diagonal tensor.
\end{theorem}
\begin{proof}The equation
$\tens{A}*\tens{B}=\tens{B}*\tens{A}$ is equivalent to
$$
\bcirc(\tens{A})\bcirc(\tens{B})=\bcirc(\tens{B})\bcirc(\tens{A}).
$$
Since $\tens{A}=\tens{P}^{-1}*\tens{D}*\tens{P}$, we have
$$
\begin{aligned}
\bcirc(\tens{A})&=(\bcirc(\tens{P}))^{-1}
\begin{bmatrix}
D^{(1)} &  D^{(p)}  & D^{(p-1)} & \cdots &  D^{(2)}\\

D^{(2)} &  D^{(1)}  &  D^{(p)} & \cdots &  D^{(3)}\\

\vdots  & \ddots& \ddots & \ddots & \vdots\\

D^{(p)} & D^{(p-1)}  &  \ddots & D^{(2)} &  D^{(1)}\\
\end{bmatrix}
(\bcirc(\tens{P}))\\
&=(\bcirc(\tens{P}))^{-1}(F_p\otimes I_n)
\begin{bmatrix}
D_1&&&\\
&D_2&&\\
&&\ddots&\\
&&&D_p
\end{bmatrix}(F_p^H\otimes I_n)\bcirc(\tens{P})\\
&=\left((F_p^H\otimes I_n)\bcirc(\tens{P})\right)^{-1}
\begin{bmatrix}
D_1&&&\\
&D_2&&\\
&&\ddots&\\
&&&D_p
\end{bmatrix}(F_p^H\otimes I_n)\bcirc(\tens{P})\\
&=P^{-1}\begin{bmatrix}
D_1&&&\\
&D_2&&\\
&&\ddots&\\
&&&D_p
\end{bmatrix} P,
\end{aligned}
$$
where $P=(F_p^H\otimes I_n)\bcirc(\tens{P})$. \par
Since $D^{(i)}$ are all diagonal matrices, by Lemma \ref{lem2-4}, we have that $D_i$ are all diagonal matrices. By the same kind of method,
$$
\begin{aligned}
\bcirc(\tens{B})&=(\bcirc(\tens{Q}))^{-1}
\begin{bmatrix}
D'^{(1)} &  D'^{(p)}  & D'^{(p-1)} & \cdots &  D'^{(2)}\\

D'^{(2)} &  D'^{(1)}  &  D'^{(p)} & \cdots &  D'^{(3)}\\

\vdots  & \ddots& \ddots & \ddots & \vdots\\

D'^{(p)} & D'^{(p-1)}  &  \ddots & D'^{(2)} &  D'^{(1)}\\
\end{bmatrix}
(\bcirc(\tens{Q}))\\
&=(\bcirc(\tens{Q}))^{-1}(F_p^H\otimes I_n)
\begin{bmatrix}
D'_1&&&\\
&D'_2&&\\
&&\ddots&\\
&&&D'_p
\end{bmatrix}(F_p\otimes I_n)\bcirc(\tens{Q})\\
&=\left((F_p\otimes I_n)\bcirc(\tens{Q})\right)^{-1}
\begin{bmatrix}
D'_1&&&\\
&D'_2&&\\
&&\ddots&\\
&&&D'_p
\end{bmatrix}(F_p\otimes I_n)\bcirc(\tens{Q})\\
&=Q^{-1}\begin{bmatrix}
D'_1&&&\\
&D'_2&&\\
&&\ddots&\\
&&&D'_p
\end{bmatrix} Q,
\end{aligned}
$$
where $Q=(F_p\otimes I_n)\bcirc(\tens{Q})$. \par
 We have $\bcirc(\tens{A})$ and $\bcirc(\tens{B})$ are both diagonalizable matrices, and $\bcirc(\tens{A})$ and $\bcirc(\tens{B})$ commutes with each other, i.e., $\bcirc(\tens{A})\bcirc(\tens{B})=\bcirc(\tens{B})\bcirc(\tens{A})$. \par
By Lemma \ref{lem2-6}, we have $\bcirc(\tens{B})$ can also be diagonalized by the matrix $P=(F_p\otimes I)\bcirc(\tens{P})$. That is,
$$
\begin{aligned}
\bcirc(\tens{B})&=P^{-1}\begin{bmatrix}
D'_1&&&\\
&D'_2&&\\
&&\ddots&\\
&&&D'_p
\end{bmatrix} P\\
&=\left((F_p\otimes I_n)\bcirc(\tens{P})\right)^{-1}
\begin{bmatrix}
D'_1&&&\\
&D'_2&&\\
&&\ddots&\\
&&&D'_p
\end{bmatrix}(F_p\otimes I_n)\bcirc(\tens{P})\\
&=(\bcirc(\tens{P}))^{-1}
\begin{bmatrix}
D'^{(1)} &  D'^{(p)}  & D'^{(p-1)} & \cdots &  D'^{(2)}\\

D'^{(2)} &  D'^{(1)}  &  D'^{(p)} & \cdots &  D'^{(3)}\\

\vdots  & \ddots& \ddots & \ddots & \vdots\\

D'^{(p)} & D'^{(p-1)}  &  \ddots & D'^{(2)} &  D'^{(1)}\\
\end{bmatrix}
(\bcirc(\tens{P}))\\
&=\bcirc(\tens{P})^{-1}\bcirc(\tens{D})\bcirc(\tens{P}).
\end{aligned}
$$
This is equivalent to $$
\tens{B}=\tens{P}^{-1}*\tens{D}'*\tens{P},
$$
where $\tens{D}'\in\mathbb{C}^{n\times n\times p}$ is also a F-diagonal tensor.
\end{proof}
By induction, we have the following generalized case.

\begin{theorem}\label{the2-7}  {\rm (Commutative tensor family)} Let $\tens{A}_1$, $\tens{A}_2$,$\cdots$, $\tens{A}_l\in\mathbb{C}^{n\times n\times p}$ be a family of  complex tensors commute with each other with T-product. If $\tens{A}_1$ is F-diagonalizable by tensor $\tens{P}\in\mathbb{C}^{n\times n\times p}$:
$$
\tens{A}_1=\tens{P}^{-1}*\tens{D}_1*\tens{P},
$$
where $\tens{D}_1\in\mathbb{C}^{n\times n\times p}$ is a F-diagonal tensor, then all the tensors $\tens{A}_1$, $\tens{A}_2$,$\ldots$, $\tens{A}_l$ are F-diagonalizable and can also be diagonalized by tensor $\tens{P}$, that is
$$
\tens{A}_i=\tens{P}^{-1}*\tens{D}_i*\tens{P},\  i=1,2,\ldots, l,
$$
where $\tens{D}_i\in\mathbb{C}^{n\times n\times p}$ is a F-diagonal tensor.
\end{theorem}

\subsection{T-characteristic and T-minimal polynomial}
In linear algebra, the characteristic polynomial of a square matrix is a polynomial which is invariant under matrix similarity relationship and has the eigenvalues as roots.  We can also introduce the concept of T-characteristic polynomial as follows:
\begin{definition}\label{def2-13} {\rm(T-characteristic polynomial)} Let $\tens{A}\in \mathbb{C}^{n\times n\times p}$ be a complex tensor, if $\bcirc(\tens{A})$ can be Fourier block diagonalized as
$$
\bcirc(\tens{A})=(F_p^H\otimes I_n)
\begin{bmatrix}
D_1&&&\\
&D_2&&\\
&&\ddots&\\
&&&D_p
\end{bmatrix}(F_p\otimes I_n),
$$
then the T-characteristic polynomial  $P_{T}(x)$ has the expression
\begin{equation}
P_{T}(x):={\rm LCM}(P_{D_1}(x),P_{D_2}(x),\cdots, P_{D_p}(x)),
\end{equation}
where `${\rm LCM}$' means the least common multiplier, $P_{D_i}(x)\ (i\in \{1,2,\ldots, p\})$ is the characteristic polynomial of the matrix $D_i$.
\end{definition}
The Cayley-Hamilton theorem states that every square matrix $A\in \mathbb{C}^{n\times n}$ over a commutative ring (such as the real or complex field) satisfies its own characteristic equation:
$$
p_A(\lambda)={\rm det}(\lambda I_n-A),
$$
which is an important theorem in matrix theories \cite{Horn1}. For the T-characteristic polynomial of tensors, we also have the following important theorem.
\begin{theorem}\label{the2-9} {\rm (Caylay-Hamilton Theorem)}
Let $\tens{A}\in \mathbb{C}^{n\times n\times p}$ be a complex tensor, $P_{T}(x)$ be the T-characteristic polynomial of $\tens{A}$. Then $\tens{A}$ satisfies the T-characteristic polynomial $P_{T}(x)$, which means
\begin{equation}
P_{T}(\tens{A})=\tens{O}.
\end{equation}
\end{theorem}
\begin{proof}
Since $P_{T}(\tens{A})$ is a tensor in $\mathbb{C}^{n\times n\times p}$, we operate `$\bcirc$' on it
$$
\begin{aligned}
\bcirc(P_{T}(\tens{A}))&=P_{T}(\bcirc(\tens{A}))\\
&=P_{T}\left(
(F_p^H\otimes I_n)
\begin{bmatrix}
D_1&&&\\
&D_2&&\\
&&\ddots&\\
&&&D_p
\end{bmatrix}(F_p\otimes I_n)
\right)\\
&=(F_p^H\otimes I_n)\begin{bmatrix}
P_{T}(D_1)&&&\\
&P_{T}(D_2)&&\\
&&\ddots&\\
&&&P_{T}(D_p)
\end{bmatrix}(F_p\otimes I_n)\\
&=(F_p^H\otimes I_n)\begin{bmatrix}
O&&&\\
&O&&\\
&&\ddots&\\
&&&O
\end{bmatrix}(F_p\otimes I_n)\\
&=O,
\end{aligned}
$$
the first step is due to Theorem \ref{the2-2}. By the $(4)$ of Lemma \ref{lem2-5}, we get $P_{T}(\tens{A})=\tens{O}$.
\end{proof}

Similarly, we define the T-minimal polynomial of $\tens{A}$ as follows.
\begin{definition}\label{def2-14} {\rm (Minimal polynomial)}
Let $\tens{A}\in \mathbb{C}^{n\times n\times p}$ be a complex tensor, if $\bcirc(\tens{A})$ can be Fourier block diagonalized as
$$
\bcirc(\tens{A})=(F_p^H\otimes I_n)
\begin{bmatrix}
D_1&&&\\
&D_2&&\\
&&\ddots&\\
&&&D_p
\end{bmatrix}(F_p\otimes I_n),
$$
then the T-minimal polynomial  $M_{T}(x)$ is defines as
\begin{equation}
M_{T}(x):={\rm LCM}(M_{D_1}(x),M_{D_2}(x),\cdots, M_{D_p}(x)),
\end{equation}
where `${\rm LCM}$' means the least common multiplier, $M_{D_i}(x)$ is the minimal polynomial of matrix $D_i$, $i\in \{1,2,\ldots, p\}$.
\end{definition}
For the T-minimal polynomial, we have the similar result as follows.
\begin{theorem}\label{the2-10}
Let $\tens{A}\in \mathbb{C}^{n\times n\times p}$ be a complex tensor, $M_{T}(x)$ be the T-minimal polynomial of $\tens{A}$. Then $\tens{A}$ satisfies the T-minimal polynomial $M_{T}(x)$, which means
\begin{equation}
M_{T}(\tens{A})=\tens{O}.
\end{equation}
\end{theorem}

We can judge whether a tensor can be F-diagonalized by the roots of its T-minimal polynomial.
\begin{corollary}\label{cor2-11}
The tensor $\tens{A}$ can be F-diagonalized if and only if the T-minimal polynomial $M_{T}(x)$ of the tensor $\tens{A}$ has no multiple roots.
\end{corollary}
\begin{proof}
By Lemma \ref{lem2-4}, the tensor $\tens{A}$ can be F-diagonalized if and only if the block matrix
$$\begin{bmatrix}
D_1&&&\\
&D_2&&\\
&&\ddots&\\
&&&D_p
\end{bmatrix}$$
can be diagonalized. Since the T-minimal polynomial $$M_{T}(x)={\rm LCM}(M_{D_1}(x),M_{D_2}(x),\cdots, M_{D_p}(x))$$ is also the minimal polynomial of the above block matrix, we obtain that the block matrix can be diagonalized if and only if $M_{T}(x)$ has no multiple roots. That is the result.
\end{proof}

We give the following corollary without proof.
\begin{corollary}\label{cor2-14}
Let $\tens{A}\in \mathbb{C}^{n\times n\times p}$ be a nilpotent tensor with nilpotent index $s\in \mathbb{Z}$. Then the T-minimal polynomial of $\tens{A}$ is $M_{T}(x)=x^s$.
\end{corollary}

\section{Conclusion}
  In this paper, by using the tensor T-product and the matrix Jordan Canonical form, we give the T-Jordan Canonical form of third order tensors. Then we give the definition of T-minimal polynomial and T-characteristic polynomials that the F-square tensors will satisfy. Cayley-Hamilton theorem also holds for tensors. We propose several tensor functions and by using the F-square tensor power series. When two F-diagonalizable tensors commutes to each other, it is proved that they can be diagonalized by the same F-square tensor. Then we obtain several kinds of tensor decomposition methods which can be viewed as the generalized cases of matrices. In the second part of our main results we focus on the T-Drazin inverse and extend the results of matrices to tensor cases such as T-index, T-Range Hermitian and so on. T-Core-nilpotent decomposition is also obtained.

\section*{Acknowledgments}
The authors would like to thank the editor and two referees for their detailed comments.
Discussions with Prof. C. Ling, Prof. Ph. Toint, Prof. Z. Huang along with his team
members, Dr. W. Ding, Dr. Z. Luo, Dr. X. Wang, and Mr. C. Mo are very helpful.

\begin{appendix}
\textbf{Appendix:}\par
\subsection{T-Polar, T-LU, T-QR and T-Schur decompositions}
Hao, Kilmer, Braman, and Hoover \cite{Hao1} introduced the tensor QR (T-QR) decomposition.
\begin{theorem}\label{the2-11}{\rm (T-QR decomposition)} {\rm \cite{Hao1}} Let $\tens{A}\in \mathbb{C}^{n\times n\times p}$ be a complex F-square tensor. Then it can be factorized as
\begin{equation}
\tens{A}=\tens{Q}*\tens{R},
\end{equation}
where $\tens{Q}\in \mathbb{C}^{n\times n\times p}$ is a unitary tensor and $\tens{R}\in \mathbb{C}^{n\times n\times p}$ is a F-upper triangular tensor.
\end{theorem}
They applied the T-QR decomposition in the facial recognition and made comparison with the traditional PCA method. Gleich, Chen, and Varah \cite{Gleich1} generalized the T-QR decomposition into circulant algebra and established its Arnoldi method.\par
Besides for the T-QR decomposition and the above T-Jordan canonical decomposition, there are other kinds of decompositions which can be introduced. In this subsection, we mainly extend the polar decompositions and LU decomposition of matrices to third order tensors. \par

We first extend the concept of (semi-)definite matrix to third order tensors.
\begin{definition}\label{def2-15}{\rm (T-Positive definite)} Let $\tens{A}\in \mathbb{C}^{n\times n\times p}$ be a complex Hermitian tensor which can be block diagonalized as
$$
\bcirc(\tens{A})=(F_p^H\otimes I_n)
\begin{bmatrix}
A_1&&&\\
&A_2&&\\
&&\ddots&\\
&&&A_p
\end{bmatrix}(F_p\otimes I_n),
$$
the matrices $A_i$ are Hermitian because $\bcirc(\tens{A})$ is a Hermitian matrix. We call $\tens{A}$ is a T-positive definite tensor if and only if all the matrices $A_i \ ( i=1,2,\cdots, p)$ are positive definite.
\end{definition}\label{cor2-15}

\begin{theorem}\label{the2-13} {\rm (T-polar decomposition)} Let $\tens{A}\in \mathbb{C}^{n\times n\times p}$ be a complex F-square Hermitian tensor. Then $\tens{A}$ can be factorized as
\begin{equation}
\tens{A}=\tens{U}*\tens{P},
\end{equation}
where $\tens{U}\in \mathbb{C}^{n\times n\times p}$ is a unitary tensor and $\tens{P}\in \mathbb{C}^{n\times n\times p}$ is a T-positive definite or T-positive semi-definite tensor.
\end{theorem}
\begin{proof}
For $\bcirc(\tens{A})$, we have
$$
\begin{aligned}
\bcirc(\tens{A})&=(F_p^H\otimes I_n)
\begin{bmatrix}
A_1&&&\\
&A_2&&\\
&&\ddots&\\
&&&A_p
\end{bmatrix}(F_p\otimes I_n).
\end{aligned}
$$
Since $\tens{A}$ is a Hermitian tensor,  all the matrices $A_i\ ( i=1,2,\cdots, p)$ are Hermitian, then they can be factorized as
$$
A_i=U_i P_i,
$$
where $U_i$ are unitary matrices and $P_i$ are positive definite matrices. It comes to
$$
\begin{aligned}
\bcirc(\tens{A})
&=(F_p^H\otimes I_n)
\begin{bmatrix}
U_1&&&\\
&U_2&&\\
&&\ddots&\\
&&&U_p
\end{bmatrix}
\begin{bmatrix}
P_1&&&\\
&P_2&&\\
&&\ddots&\\
&&&P_p
\end{bmatrix}
(F_p\otimes I_n)\\
&=\bcirc(\tens{U})\bcirc(\tens{P}),
\end{aligned}
$$
which is equivalent to
$$
\tens{A}=\tens{U}*\tens{P},
$$
where $\tens{U}$ and $\tens{P}$ are unitary tensor and positive tensor, respectively.
\end{proof}
Another important matrix decomposition is called the LU decomposition which is proved to be of great importance in numerical linear algebra.

We also find the similar factorization hold for third order tensors.
\begin{theorem}\label{the2-14}{\rm (T-LU decomposition)} Let $\tens{A}\in \mathbb{C}^{n\times n\times p}$ be a complex F-square tensor. Then it can be factorized as
\begin{equation}
\tens{A}=\tens{L}*\tens{U},
\end{equation}
where $\tens{L}\in \mathbb{C}^{n\times n\times p}$ is a F-lower tensor and $\tens{U}\in \mathbb{C}^{n\times n\times p}$ is a F-upper tensor.
\end{theorem}
\begin{proof}
For matrix $\bcirc(\tens{A})$, we have
$$
\begin{aligned}
\bcirc(\tens{A})&=(F_p^H\otimes I_n)
\begin{bmatrix}
A_1&&&\\
&A_2&&\\
&&\ddots&\\
&&&A_p
\end{bmatrix}(F_p\otimes I_n),
\end{aligned}
$$
each $A_i \ ( i=1,2,\cdots, p)$ is a square matrix which can be factorized as
$$A_i=L_i U_i,$$
where $U_i$ are upper triangular matrices and $L_i$ are lower triangular matrices. It comes to
$$
\begin{aligned}
\bcirc(\tens{A})
&=(F_p^H\otimes I_n)
\begin{bmatrix}
L_1&&&\\
&L_2&&\\
&&\ddots&\\
&&&L_p
\end{bmatrix}
\begin{bmatrix}
U_1&&&\\
&U_2&&\\
&&\ddots&\\
&&&U_p
\end{bmatrix}
(F_p\otimes I_n)\\
&=\bcirc(\tens{L})\bcirc(\tens{U}),
\end{aligned}
$$
which is equivalent to
$$
\tens{A}=\tens{L}*\tens{U}.
$$
Here $\tens{L}$ and $\tens{U}$ are respectively F-upper tensor and F-lower tensor by Lemma \ref{lem2-4}.
\end{proof}
By the same kind of method, we obtain the following theorem:
\begin{theorem}\label{the2-15}{\rm (T-Schur decomposition)} Let $\tens{A}\in \mathbb{C}^{n\times n\times p}$ be a complex F-square tensor. Then it can be factorized as
\begin{equation}
\tens{A}=\tens{Q}^{-1}*\tens{T}*\tens{Q},
\end{equation}
where $\tens{Q}\in \mathbb{C}^{n\times n\times p}$ is a unitary tensor and $\tens{T}\in \mathbb{C}^{n\times n\times p}$ is a F-upper triangular tensor.
\end{theorem}

\subsection{Tensor T-index}
In the following subsections, we will investigate a special kind of generalized inverse having spectral properties besides for the Moore-Penrose inverse defined in the above subsections. Only F-square third order tensors will be considered, since only they have T-eigenvalues.\par
If a tensor $\tens{A}\in \mathbb{C}^{n\times n\times p}$ is invertible, it is easy to see that all the T-eigenvalues $\lambda_i\ (i=1,2,\ldots, np)$ are non-zero and the T-eigenvalues of $\tens{A}^{-1}$ are $\lambda_i^{-1}\ (i=1,2,\ldots, np)$.\par

The Moore-Penrose inverse of irreversible third order tensors can be defined as the unique solution of tensor $\tens{X}\in \mathbb{C}^{n\times n\times p}$ satisfying the following four equations,
\begin{equation}
\tens{A}*\tens{X}*\tens{A}=\tens{A},
\quad
\tens{X}*\tens{A}*\tens{X}=\tens{X},
\quad
(\tens{A}*\tens{X})^H=\tens{A}*\tens{X},
\quad
(\tens{X}*\tens{A})^H=\tens{X}*\tens{A},
\end{equation}
where $\tens{A}^H$ denotes the conjugate transpose of tensor $\tens{A}$ and the unique solution of the equation is denoted by $\tens{A}^{\dag}$. As a special case, if $\ten{A}$ is invertible, then $\tens{X}=\tens{A}^{-1}$ trivially satisfies the above four equations, which means the Moore-Penrose inverse of an invertible tensor is the inverse of the tensor.\par
Similarly, the T-Drazin inverse $\tens{A}^{D}$ (or called $\{1^k,2,5\}$-inverse) is defined as
\begin{equation}
\tens{A}*\tens{X}*\tens{A}=\tens{A},
\quad
\tens{X}*\tens{A}*\tens{X}=\tens{X},
\quad
\tens{A}^k*\tens{X}*\tens{A}=\tens{A}^k,
\end{equation}
where $k\in \mathbb{Z}$ is a given positive integer. The T-group inverse $\tens{A}^{\#}$ (or called $\{1,2,5\}$-inverse) is defined as
\begin{equation}
\tens{A}*\tens{X}*\tens{A}=\tens{A},
\quad
\tens{X}*\tens{A}*\tens{X}=\tens{X},
\quad
\tens{A}*\tens{X}=\tens{X}*\tens{A}.
\end{equation}

%It can be seen that the set of the equations $(17)$ is equivalent to the set
%\begin{equation}
%\tens{A}*\tens{X}=\tens{X}*\tens{A},
%\quad
%\tens{A}^{k+1}*\tens{X}=\tens{A}^k,
%\quad
%\tens{A}*\tens{X}^2=\tens{X}.
%\end{equation}
In matrix case, the smallest positive integer $k$ for which
$$
{\rm rank} (A^k)={\rm rank} (A^{k+1})
$$
holds is called the index of $A\in\mathbb{C}^{n\times n}$, usually denoted as ${\rm Ind}(A)$ \cite{Ben1}. For third order tensors, we can also introduce the similar concept with the T-product.
\begin{definition}\label{def2-16} {\rm (T-index)} Let $\tens{A}\in\mathbb{C}^{n\times n\times p}$ be a complex tensor. The T-index of tensor $\tens{A}$ is defined as
\begin{equation}
{\rm Ind}_T (\tens{A})={\rm Ind}(\bcirc(\tens{A})).
\end{equation}
\end{definition}

\begin{definition}\label{def2-17} {\rm (T-rank)} Let $\tens{A}\in\mathbb{C}^{n\times n\times p}$ be a complex tensor. The T-rank of tensor $\tens{A}$ is defined as
\begin{equation}
{\rm rank}_T (\tens{A})={\rm rank}(\bcirc(\tens{A})).
\end{equation}
\end{definition}

\begin{corollary}\label{cor2-17} Let $\tens{A}\in \mathbb{C}^{n\times n\times p}$ be a complex F-square tensor satisfies
$$
\begin{aligned}
\bcirc(\tens{A})&=(F_p^H\otimes I_n)
\begin{bmatrix}
A_1&&&\\
&A_2&&\\
&&\ddots&\\
&&&A_p
\end{bmatrix}(F_p\otimes I_n),
\end{aligned}
$$
then
\begin{equation}
\ind_T(\tens{A})=\max_{1\leq i\leq p}\{(\ind(A_i)\}.
\end{equation}
\end{corollary}
\begin{proof}
Since
$$
\begin{aligned}
\bcirc(\tens{A})^k&=(F_p^H\otimes I_n)
\left(\begin{bmatrix}
A_1&&&\\
&A_2&&\\
&&\ddots&\\
&&&A_p
\end{bmatrix}\right)^k(F_p\otimes I_n),
\end{aligned}
$$
The index of the middle block matrix is $\max\limits_{1\leq i\leq p}\{\ind(A_i))\}$, which ends the proof.
\end{proof}

By recalling the definition of range space null space of third order tensors, we have the following lemma.
\begin{lemma}\label{lem2-7}
Let $\tens{A}\in \mathbb{C}^{n\times n\times p }$ and $\ind_T(\tens{A})=k$. Then we have:\\
{\rm (1)} All tensors $\{\tens{A}^l, l\geq k\}$ have the same T-rank, the same T-range $\tens{R}(\tens{A}^l)$ and the same T-null space $\tens{N}(\tens{A}^l)$,\\
{\rm (2)} All tensors $\{(\tens{A}^{\top})^l, l\geq k\}$ have the same T-rank, the same T-range $\tens{R}((\tens{A}^{\top})^l)$ and the same T-null space $\tens{N}((\tens{A}^{\top})^l)$,\\
{\rm (3)} All tensors $\{(\tens{A}^{H})^l, l\geq k\}$ have the same T-rank, the same T-range $\tens{R}((\tens{A}^{H})^l)$ and the same T-null space $\tens{N}((\tens{A}^{H})^l)$.\\
{\rm (4)} For no $l$ less than $k$ do $\tens{A}^l$ and a higher power of $\tens{A}$ (or their transposes or conjugate transposes) have the same T-range or the same T-null space.
\end{lemma}

The relationship between the tensor T-index and the tensor T-minimal polynomial is as follows.
\begin{theorem}\label{the2-16}
Let $\tens{A}\in \mathbb{C}^{n\times n\times p }$ be a complex third order tensor. Then the following statements are equivalent:\\
{\rm (1)} $\ind_T(\tens{A})=k$.\\
{\rm (2)} The smallest integer holds for equation $\tens{A}^k*\tens{X}*\tens{A}=\tens{A}^k$ is $k$.\\
{\rm (3)} $k$ is the multiplicity of $\lambda=0$ as a zero point of the T-minimal polynomial $M_T(x)$.
\end{theorem}
\begin{proof}
${\rm (1)} \Longleftrightarrow {\rm (2)}$ Since $\rank_T(\tens{A}^{l+1})=\rank_T(\tens{A}^l)$ is equivalent to $\tens{R}(\tens{A}^{l+1})=\tens{R}(\tens{A}^l)$, which means there exists a tensor $\tens{X}$ subject to $\tens{A}^{l+1}*\tens{X}=\tens{A}^l$.\\
${\rm (2)} \Longleftrightarrow {\rm (3)}$ Let the T-minimal polynomial have the factorization,
$$
M_{T}(\lambda)=\lambda^l p(\lambda),
$$
where $p(0)\neq 0$. If $k$ satisfies {\rm (2)}, we need to show that $k=l$. We have
$$
p(\tens{A})*\tens{A}^l=\tens{O}.
$$
If $l>k$, then we have
$$
\tens{O}=p(\tens{A})*\tens{A}^l*\tens{X}=p(\tens{A})*\tens{A}^{l-1},
$$
where $\lambda^{l-1}p(\lambda)$ is of lower degree than $m(\lambda)$, contrary to the definition of the minimal polynomial.\par
$$
M_T(\lambda)=c\lambda^l(1-\lambda q(\lambda)),
$$
where $c\neq 0$ and $q(\lambda)$ is a polynomial. Then we have
$$
\tens{A}^{l+1}*q(\tens{A})=\tens{A}^l.
$$
If $l<k$, then it has contradiction to ${\rm (2)}$.
\end{proof}

\subsection{T-Drazin inverse}
We don't want to dismiss the concept of generalized inverse of tensors only to F-diagonalizable tensors. In this subsection, we will investigate the existence and properties of these kinds of inverses satisfying the equations $(16)$ and $(17)$.

%We first consider the $\{1,2,5\}$-inverse of $\tens{A}$, that is the tensor $\tens{X}$ %satisfies
%$$
%\tens{A}*\tens{X}*\tens{A}=\tens{A},
%\quad
%\tens{X}*\tens{A}*\tens{X}=\tens{X},
%\quad
%\tens{A}*\tens{X}=\tens{X}*\tens{A}.
%$$
%For matrix cases, we have the following lemma¡£
%\begin{lemma}\label{lem2-8}
%A square matrix $A$ has a group inverse if and only if $\ind(A)=1$, i.e.,
%$$
%\rank(A)=\rank(A^2).
%$$
%When the group inverse exists, it is unique.
%\end{lemma}
%By using Lemma \ref{lem2-8}, we can get the existence and uniqueness of third order tensors:
%\begin{theorem}\label{the2-17} {\rm (Group inverse)} Let $\tens{A}\in \mathbb{C}^{n\times %n\times p}$ be a complex tensor. Then it has a group inverse (satisfying the above three %equations) if and only if $\ind_T(\tens{A})=1$
%$$
%\rank_T(\tens{A})=\rank_T(\tens{A}^2).
%$$
%When the group inverse exists, it is unique.
%\end{theorem}

%We denote the group inverse of tensor $\tens{A}$ by $\tens{A}^{\#}$.
For some kind class of tensors, the group inverse and the Moore-Penrose inverse are the same.
\begin{definition}\label{def2-18} {\rm (T-Range Hermitian)} Let $\tens{A}\in \mathbb{C}^{n\times n\times p}$ be a complex tensor. It is called T-range Hermitian if and only if \begin{equation}
\tens{R}(\tens{A})=\tens{R}(\tens{A}^H),
\end{equation}
or equivalently, if and only if
$$
\tens{N}(\tens{A})=\tens{N}(\tens{A}^H).
$$
\end{definition}
By the same kind of methods in matrices \cite{Ben1}, we get the following result.
\begin{theorem}\label{the2-18}
$\tens{A}^{\#}=\tens{A}^{\dag}$ if and only if $\tens{A}$ is T-range Hermitian.
\end{theorem}
\begin{proof}
By the above proof, we find
$$
\bcirc(\tens{A}^{\#})=(F_p^H\otimes I_n)
\begin{bmatrix}
A_1^{\#}&&&\\
&A_2^{\#}&&\\
&&\ddots&\\
&&&A_p^{\#}
\end{bmatrix}(F_p\otimes I_n).
$$
Also, for Moore-Penrose inverse of $\tens{A}$, we have
$$
\bcirc(\tens{A}^{\dag})=(F_p^H\otimes I_n)
\begin{bmatrix}
A_1^{\dag}&&&\\
&A_2^{\dag}&&\\
&&\ddots&\\
&&&A_p^{\dag}
\end{bmatrix}(F_p\otimes I_n).
$$
Since $\tens{R}(\tens{A})=\tens{R}(\tens{A}^H)$ and $(F_p\otimes I_n)$ is an invertible matrix, we have
$$
{\rm Range}(A_1)\oplus{\rm Range}(A_2)\oplus\cdots\oplus{\rm Range}(A_p)={\rm Range}(A_1^H)\oplus{\rm Range}(A_2^H)\oplus\cdots\oplus{\rm Range}(A_p^H).
$$
Each matrix $A_i \ (i=1,2,\ldots, p)$ is range Hermitian.  $A_i^{\dag}=A_i^{\#}\ (i=1,2,\ldots, p)$, which comes to the result.
\end{proof}
\begin{lemma}\label{lem2-9}
Let $\tens{A}\in\mathbb{C}^{n\times n\times p}$ be a T-range Hermitian complex tensor with T-index $\ind_T(\tens{A})=1$. Let the T-Jordan canonical form of $\tens{A}$ be
$$
\tens{A}=\tens{P}^{-1}*\tens{J}*\tens{P},
$$
Then the T-index of $\tens{J}$
$$
\ind_T(\tens{J})=1.
$$
\end{lemma}
\begin{proof}
By Corollary \ref{cor2-1}, we have $\tens{A}^k=\tens{P}^{-1}*\tens{J}^k*\tens{P}$. Then it comes to
$$
\bcirc(\tens{J}^k)=(F_p^H\otimes I_n)
\begin{bmatrix}
J_1^{k}&&&\\
&J_2^{k}&&\\
&&\ddots&\\
&&&J_p^{k}
\end{bmatrix}(F_p\otimes I_n),
$$
each $J_i^k$ can be partitioned in the form
$$
J_i=\begin{bmatrix}
J_i^{1}&0\\
0&J_i^{0}
\end{bmatrix},
$$
where $J_i^{1}$ is the nonsingular block with non-zero T-eigenvalues while $J_i^{0}$ is the nilpotent part with zero T-eigenvalues. It is easy to see $\rank(J_i^{1})=\rank((J_i^{1})^2)$. From the Jordan structure of the Jordan form $\rank(J_i^{0})<\rank((J_i^{0})^2)$ unless $J_i^{0}$ is the null matrix $O$. Since $\rank_T(\tens{A})=\rank_T((\tens{A})^{2})$, so it comes to each block $J_i^{0}$ is the null matrix $O$.
$\rank_T((\tens{J})^k)=\rank_T((\tens{J})^{k+1})$ for all $k\geq 1$, which means $\ind_T(\tens{J})=1$.
\end{proof}
By the above Theorem \ref{the2-18} and Lemma \ref{lem2-9}, we have the following theorem.
\begin{theorem}\label{the2-19}
Let $\tens{A}\in\mathbb{C}^{n\times n\times p}$ be a complex tensor with T-index $1$. Let the T-Jordan canonical form of $\tens{A}$ is
$$
\tens{A}=\tens{P}^{-1}*\tens{J}*\tens{P},
$$
where $\tens{P}$ is an invertible tensor and $\tens{J}$ is the Jordan canonical form of $\tens{A}$. Then
\begin{equation}
\tens{A}^{\#}=\tens{P}^{-1}*\tens{J}^{\#}*\tens{P},
\end{equation}
where $\tens{J}^{\#}$ is the group inverse of $\tens{J}$.
\end{theorem}
%\begin{proof}
%It can be verified that
%$$
%\tens{A}^{\#}=\tens{P}^{-1}*\tens{J}^{\#}*\tens{P},
%$$
%since the T-index of $\tens{A}$ is $1$, then the T-index of $\tens{J}$ is also $1$ because %$\tens{P}$ is invertible. So by the Theorem \ref{the2-18}, we have
%$$
%\tens{J}^{\dag}=\tens{J}^{\#},
%$$
%then it comes to the proof.
%\end{proof}

Now we consider the $\{1^k,2,5\}$-inverse $\tens{X}$, that is
$$
\tens{A}^k*\tens{X}*\tens{A}=\tens{A}^k,
\quad
\tens{X}*\tens{A}*\tens{X}=\tens{X}
\quad
\tens{A}*\tens{X}=\tens{X}*\tens{A}.
$$
We call $\tens{X}$ satisfying the above three equations T-Drazin inverse of $\tens{A}$, denoted by $\tens{A}^D$. By the same kind of construction as group inverse, we directly give the expression of T-Drazin inverse.
\begin{theorem}\label{the2-20}{\rm (T-Drazin inverse)} Let $\tens{A}\in \mathbb{C}^{n\times n\times p}$ be a complex tensor which has the T-Jordan canonical form
$$
\tens{A}=\tens{P}^{-1}*\tens{J}*\tens{P}.
$$
Then the T-Drazin inverse is given by
\begin{equation}
\tens{A}^D=\tens{P}^{-1}*\tens{J}^D*\tens{P},
\end{equation}
where $\tens{J}^D$ is given as follows. Suppose the matrix $\bcirc(\tens{J})$ has decomposition,
$$
\bcirc(\tens{J})=(F_p^H\otimes I_n)
\begin{bmatrix}
J_1&&&\\
&J_2&&\\
&&\ddots&\\
&&&J_p
\end{bmatrix}(F_p\otimes I_n),
$$
each block $J_i$ can be partitioned as
$
J_i=\begin{bmatrix}
J_i^{1}&0\\
0&J_i^{0}
\end{bmatrix},
$
denote $J_i^D=
J_i=\begin{bmatrix}
(J_i^{1})^{-1}&0\\
0&0
\end{bmatrix}
$ to be the Drazin inverse of the matrix $J_i$.
Then we define the T-Drazin inverse $\tens{J}^D$ as
$$
\bcirc(\tens{J}^D)=(F_p^H\otimes I_n)
\begin{bmatrix}
J_1^D&&&\\
&J_2^D&&\\
&&\ddots&\\
&&&J_p^D
\end{bmatrix}(F_p\otimes I_n).
$$
This is the unique solution to the above three equations and the T-group inverse is the particular case of T-Drazin inverse for tensors with T-index $1$.
\end{theorem}
By the construction of T-Drazin inverse, we have the following corollary.
\begin{corollary}\label{cor2-20}
The T-Drazin inverse preserves similarity: Let $\tens{A}\in\mathbb{C}^{n\times n\times p}$ be a complex tensor. If $\tens{X}\in\mathbb{C}^{n\times n\times p}$ be a nonsingular tensor and
$$
\tens{A}=\tens{X}^{-1}*\tens{B}*\tens{X},
$$
then
$$
\tens{A}^D=\tens{X}^{-1}*\tens{B}^D*\tens{X}, \quad
\tens{A}^D*(\tens{A}^D)^{\#}=\tens{A}*\tens{A}^D.
$$
\end{corollary}
%\begin{corollary}\label{cor2-21}
%Let $\tens{A}\in\mathbb{C}^{n\times n\times p}$ be a complex tensor. Then
%$$
%\tens{A}^D*(\tens{A}^D)^{\#}=\tens{A}*\tens{A}^D.
%$$
%\end{corollary}

\subsection{T-core-nilpotent decomposition}
There is another important factorization of tensors based on the T-Drazin inverse, and that is the T-core-nilpotent decomposition.
\begin{definition}\label{def2-20} {\rm (T-Core)} Let $\tens{A}\in\mathbb{C}^{n\times n\times p}$ be a complex tensor. The product
$$
\tens{C}_{\tens{A}}=\tens{A}*\tens{A}^D*\tens{A}=\tens{A}^2*\tens{A}^D=\tens{A}^D*\tens{A}^2
$$
is called the T-core of tensor $\tens{A}$.
\end{definition}
Intuitively, the T-core contains the most of the basic structure of tensor $\tens{A}$. If the T-core is removed from the tensor, then not much information of $\tens{A}$ will remain. The next theorem shows in what sense this is true.
\begin{theorem}\label{the2-25} Let $\tens{A}\in\mathbb{C}^{n\times n\times p}$ be a complex tensor, $\tens{C}_{\tens{A}}$ is the T-core of tensor $\tens{A}$. Then
$$
\tens{N}_{\tens{A}}=\tens{A}-\tens{C}_{\tens{A}}
$$
is a nilpotent tensor of T-index $k=\ind_T(\tens{A})$.
\end{theorem}
\begin{proof}
If $\ind_T(\tens{A})=\tens{O}$, then it is invertible, and $\tens{N}_{\tens{A}}=\tens{O}$ is of T-index $0$. So we assume $\ind_T(\tens{A})\geq 1$, since
$$
(\tens{N}_{\tens{A}})^k=(\tens{A}-\tens{A}*\tens{A}^D*\tens{A})^k=\tens{A}^k *(\tens{I}
-\tens{A}*\tens{A}^D)=\tens{O}.
$$
On the other hand,
$$
\tens{A}^l-\tens{A}^{l+1}*\tens{A}^D\neq \tens{O}, \quad l<k,
$$
it comes to $\ind_T(\tens{N}_{\tens{A}})=k$.
\end{proof}

\begin{definition}\label{def2-21} {\rm (T-core-nilpotent decompotion)} Let $\tens{A}\in\mathbb{C}^{n\times n\times p}$ be a complex tensor, we call
$$
\tens{N}_{\tens{A}}=\tens{A}-\tens{C}_{\tens{A}}=(\tens{I}-\tens{A}*\tens{A}^D)*\tens{A}
$$
the T-nilpotent part of $\tens{A}$ and
$$
\tens{A}=\tens{C}_{\tens{A}}+\tens{N}_{\tens{A}}
$$
the T-core-nilpotent decomposition of $\tens{A}$.
\end{definition}
Now we give the construction of T-core-nilpotent decomposition of a tensor. Suppose tensor $\tens{A}\in\mathbb{C}^{n\times n\times p}$ is of T-index $\ind_T(\tens{A})=k$ and has the T-Jordan decomposition $\tens{A}=\tens{P}^{-1}*\tens{J}*\tens{P}$, where
$$
\bcirc(\tens{J})=(F_p^H\otimes I_n)
\begin{bmatrix}
J_1&&&\\
&J_2&&\\
&&\ddots&\\
&&&J_p
\end{bmatrix}(F_p\otimes I_n),
$$
where each $J_i$ can be block partitioned as
$$
J_i=\begin{bmatrix}
C_i&0\\
0&N_i
\end{bmatrix}=\begin{bmatrix}
C_i&0\\
0&0
\end{bmatrix}+\begin{bmatrix}
0&0\\
0&N_i
\end{bmatrix}=J_i^C+J_i^N,
$$
where $C_i$ is a non-singular matrix and $N_i$ is nilpotent with
$$
\max_{1\leq i\leq p}\{\ind(N_i)\}=k,
$$
then $\bcirc(\tens{J})=\bcirc(\tens{J}^C)+\bcirc(\tens{J}^N)$, that is
$$
\tens{A}=\tens{P}^{-1}*\tens{J}*\tens{P}=\tens{P}^{-1}*(\tens{J}^C+\tens{J}^N)*\tens{P}=\tens{C}_{\tens{A}}+\tens{N}_{\tens{A}},
$$
which is the construction of T-core-nilpotent decomposition of $\tens{A}$.
By combining the T-core-nilpotent decomposition and T-index 1-nilpotent decomposition, we find $\tens{B}=(\tens{A}^D)^{\#}$ is the T-core of tensor $\tens{A}$ and the T-core-nilpotent decomposition is unique.

The following theorem tells when the T-Drazin inverse reduces to the T-Moore-Penrose inverse.
\begin{corollary}\label{cor2-25}
Let $\tens{A}\in\mathbb{C}^{n\times n\times p}$ be a complex tensor. Then $\tens{A}^D=\tens{A}^{\dag}$ if and only if $\tens{A}*\tens{A}^{\dag}=\tens{A}^{\dag}*\tens{A}$¡£
\end{corollary}
\begin{proof}
If $\tens{A}*\tens{A}^{\dag}=\tens{A}^{\dag}*\tens{A}$, then $\tens{A}^{\dag}$ is the $\{1,2,5\}$-inverse of $\tens{A}$, then $\tens{A}^{\dag}=\tens{A}^{\#}=\tens{A}^D$. Conversely, if $\tens{A}^D=\tens{A}^{\dag}$, then $\tens{A}*\tens{A}^{\dag}=\tens{A}*\tens{A}^D=\tens{A}^D*\tens{A}=\tens{A}^{\dag}*\tens{A}$.
\end{proof}
It can be shown that the T-Drazin inverse can be expressed by a limiting formula.
\begin{definition}\label{def2-22}
Let $\tens{A}\in\mathbb{C}^{n\times n\times p}$ be a complex tensor, $\tens{C}_{\tens{A}}$ be the T-core of $\tens{A}$ and $\tens{N}_{\tens{A}}$ be the T-nilpotent of $\tens{A}$. For integers $m\geq -1$, we define
$$
\tens{C}_{\tens{A}}^{(m)}=\tens{A}^{m+1}*\tens{A}^D=\begin{cases}
\tens{A}^D, \quad  &m=-1,\\
\tens{A}*\tens{A}^D,\quad  &m=0,\\
\tens{C}_{\tens{A}}^m, \quad  &m\geq 1.
\end{cases}
$$
and
$$
\tens{N}_{\tens{A}}^{(m)}=\begin{cases}
\tens{O}, \quad  &m=-1,\\
\tens{A}^m-\tens{C}_{\tens{A}}^{(m)},\quad &m\geq 0,
\end{cases}=
\begin{cases}
\tens{O}, \quad &m=-1,\\
\tens{I}-\tens{A}*\tens{A}^D,\quad &m=0,\\
\tens{N}_{\tens{A}}^m, \quad &m\geq 1.
\end{cases}
$$
\end{definition}
\begin{theorem}\label{the2-26} Let $\tens{A}\in\mathbb{C}^{n\times n\times p}$ be a complex tensor and $\ind_T(\tens{A})=k$. For every integer $l\geq k$,
$$
\tens{A}^D=\lim_{z\rightarrow 0} (\tens{A}^{l+1}+z\tens{I})^{-1}*\tens{A}^l.
$$
For every integer $l\geq 0$,
$$
\tens{A}^D=\lim_{z\rightarrow 0} (\tens{A}^{l+1}+z\tens{I})^{-1}*\tens{C}_{\tens{A}}^{(l)}.
$$
\end{theorem}
\begin{proof}
For non-singular tensor $\tens{C}$, we have
$$
\lim_{z\rightarrow 0} (\tens{C}^{l+1}+z\tens{I})^{-1}*\tens{C}^l=\tens{C}^{-1},
$$
combining with $\tens{C}_{\tens{A}}^{(l)}=\tens{A}^{l+1}*\tens{A}^D=\tens{A}^l$ for $l\geq k$ and the T-Jordan decomposition, we will come to the result.
\end{proof}
\begin{corollary}\label{cor2-26}
Let $\tens{A}\in\mathbb{C}^{n\times n\times p}$ be a complex tensor, we have
$$
\tens{A}^D=\lim_{z\rightarrow 0} (\tens{A}^{n+1}+z\tens{I})^{-1}*\tens{A}^n
$$
\end{corollary}
\begin{proof}
It comes from the tensor T-index $\ind_T(\tens{A})=k\leq n$.
\end{proof}
The T-index of a tensor $\tens{A}$ can also be obtained by a limited process. We need the following two lemmas.
\begin{lemma}\label{lem2-10} Let $\tens{A}\in\mathbb{C}^{n\times n\times p}$ be an irrevertible complex tensor. A positive integer $p$ satisfies $\ind_T(\tens{A}^p)=1$ if and only if $p\geq \ind_T(\tens{A})$. Equivalently, the smallest positive integer $l$ for which $\ind_T(\tens{A}^l)=1$ is the T-index of tensor $\tens{A}$.
\end{lemma}
\begin{lemma}\label{lem2-11} Let $\tens{N}\in\mathbb{C}^{n\times n\times p}$ be a nilpotent tensor with T-index $\ind_T(\tens{N})=k$. Suppose $m$ and $q$ be non-negative numbers, the limit
$$
\lim_{z\rightarrow 0} z^m(\tens{N}+z\tens{I})^{-1}*\tens{N}^q
$$
exists if and only if $m+q\geq k$. When the limit exists, its value is given by
$$
\lim_{z\rightarrow 0} z^m(\tens{N}+z\tens{I})^{-1}*\tens{N}^q=\begin{cases}
(-1)^{m+1}\tens{N}^{m+q-1},\quad &m>0,\\
\tens{O}, \quad &m=0.
\end{cases}
$$
\end{lemma}
\begin{proof}
If $\tens{N}=\tens{O}$, then $\ind_T(\tens{N})=1$. The limited process degenerate to
$$
\lim_{z\rightarrow 0} z^{m-1}\tens{O}^q=\begin{cases}
\lim\limits_{z\rightarrow 0} z^{m-1}\tens{I},\quad &q=0,\\
0,\quad &q\geq 1.
\end{cases}
$$
The limit exists if and only if $q\geq 1$ or $m\geq 1$, which is equivalent to $m+q\geq 1$.\par
If $\tens{N}\neq \tens{O}$, we have the Laurant expansion
$$
(\tens{N}+z\tens{I})^{-1}=\sum_{i=0}^{k-1}(-1)^i\frac{\tens{N}^i}{z^{i+1}}.
$$
Then it comes to be
$$
\begin{aligned}
z^m(\tens{N}+z\tens{I})^{-1}*\tens{N}^q=&z^{m-1}\tens{N}^q-z^{m-2}\tens{N}^{q+1}+\cdots+(-1)^{m-2}z\tens{N}^{m+q-2}\\
&+(-1)^{m-1}z\tens{N}^{m+q-1}+\frac{(-1)^{m}\tens{N}^{m+q}}{z}+\cdots\\
&+\frac{(-1)^{k-1}\tens{N}^{q+k-1}}{z^{k-m}}.
\end{aligned}
$$
If $m+q\geq k$, then the limit exists. Conversely, if the limit exists, then from $\tens{N}^{m+q}=\tens{O}$ we have $m+q\geq k$.
\end{proof}
By the above two lemmas, we have the following theorem.
\begin{theorem}\label{the2-27} Let $\tens{A}\in\mathbb{C}^{n\times n\times p}$ be a complex tensor with T-index $\ind_T(\tens{A})=k$. For non-negative integers $m$ and $q$, the limit
$$
\lim_{z\rightarrow 0}z^m(\tens{A}+z\tens{I})^{-1}*\tens{A}^q
$$
exists if and only if $m+q\geq k$, in which case the limit is given by
\begin{equation}
\lim_{z\rightarrow 0}z^m(\tens{A}+z\tens{I})^{-1}*\tens{A}^q=\begin{cases}
(-1)^{m+1}(\tens{I}-\tens{A}*\tens{A}^D)*\tens{A}^{m+q-1}, \quad &m>0,\\
\tens{A}^D*\tens{A}^q, \quad &m=0.
\end{cases}
\end{equation}
\end{theorem}
\end{appendix}

\end{document}